\documentclass[12pt, a4paper]{amsart}
\usepackage{amsmath}
\usepackage{geometry,amsthm,graphics,tabularx,amssymb,
shapepar}
\usepackage{amscd}
\usepackage[usenames]{color}

\usepackage{graphicx}
\usepackage[all]{xy}
\newcommand{\GL}{{\mathrm{GL}}}

\newcommand{\Hom}{{\mathrm{Hom}}}

\newcommand{\rk}{{\mathrm{k}}}

\newcommand{\Spec}{{\mathrm{Spec}}}

\renewcommand{\rk}{\mathrm k}

\newcommand{\C}{\mathbb{C}}
\newcommand{\R}{\mathbb R}

\newcommand{\be}{\begin {equation}}
\newcommand{\ee}{\end {equation}}
\newcommand{\bee}{\begin {equation*}}
\newcommand{\eee}{\end {equation*}}

\newcommand{\cf}{\emph{cf.}~}

\theoremstyle{Theorem}

\theoremstyle{Theorem}
\newtheorem{lem}{Lemma}[section]
\newtheorem{corl}[lem]{Corollary}
\newtheorem{thml}[lem]{Theorem}

\theoremstyle{Theorem}
\newtheorem{prp}{Proposition}[section]
\newtheorem{corp}[prp]{Corollary}
\newtheorem{lemp}[prp]{Lemma}
\newtheorem{thmp}[prp]{Theorem}

\theoremstyle{Plain}

\theoremstyle{Definition}
\newtheorem{dfn}{Definition}[section]

\newtheorem{dfnl}[lem]{Definition}

\newtheorem{thmd}[dfn]{Theorem}

\begin{document}

\title[Chevalley's theorem]{Chevalley's theorem for affine Nash groups}

\author[Y. Fang]{Yingjue Fang}

\address{College of Mathematics and Computational Science, Shenzhen University, Shenzhen, 518060, China}
\email{joyfang@szu.edu.cn}

\author [B. Sun] {Binyong Sun}
\address{Academy of Mathematics and Systems Science\\
Chinese Academy of Sciences\\
Beijing, 100190,  China} \email{sun@math.ac.cn}

\subjclass[2000]{22E15, 14L10, 14P20}

\keywords{Nash manifold, Nash group, algebraic group, Chevalley's theorem}
\thanks{B. Sun was  supported by NSFC Grants 11222101 and 11321101.}

\begin{abstract}
We formulate and prove Chevalley's theorem in the setting of affine Nash groups. As a consequence, we show that the semi-direct product of two almost linear Nash groups is still an almost linear Nash group.
\end{abstract}

 \maketitle
%\tableofcontents

\section{Introduction}

The reader is referred to \cite{BCR, Sh} for basic notions concerning  Nash manifolds and Nash maps. See also \cite[Section 2]{Sun}.
Recall that a Nash group is a group which is simultaneously a Nash manifold so that all group operations are Nash maps.  A Nash manifold is said to be affine if it is Nash diffeomorphic to a Nash submanifold of some finite dimensional real vector spaces. It is known that
every affine Nash manifold is actually Nash diffeomorphic to a closed Nash
submanifold of some finite dimensional real vector spaces (\cf \cite[Section 2.22]{Sh2}). A Nash group is said to be affine if it is affine as a Nash manifold. Thanks to the work of  E. Hrushovski and A. Pillay (see Theorem \ref{main30}), we know that affine Nash groups are closely related to real algebraic groups.

Chevalley's theorem (see Theorem \ref{mainch}) is a fundamental result in the structure theory of algebraic groups. In this note, we will formulate and prove an analogue of Chevalley's theorem in the setting of affine Nash groups.

Recall from \cite{Sun} that a Nash group is said to be almost linear if there exists a Nash homomorphism with finite kernel from it to the Nash group $\GL_n(\R)$, for some $n\geq 0$. Using \cite[Proposition III.1.7]{Sh}, we know that every almost linear Nash group is affine.  Almost linear Nash groups provide a very convenient setting for the  study of infinite dimensional smooth representations (\cf  \cite{AGKL, AGS, du, SZ}). The structure theory of almost linear Nash groups is systematically studied in \cite{Sun}. On the other hand, we introduce the following definition.

\begin{dfn}
An affine Nash group is said to be complete if it has no non-trivial connected almost linear Nash subgroup. An abelian Nash manifold is a connected, complete affine Nash group.
\end{dfn}

Similar to abelian varieties,  every abelian Nash manifold is  compact and commutative  as a Lie group (see Corollary \ref{abncc}). But a connected, compact, abelian, affine Nash group is not necessarily an abelian Nash group.  For example, the group of complex numbers of modulus one is naturally an affine Nash group. It is almost linear, and is not an abelian Nash manifold.

In Proposition \ref{quotient0}, we will prove that the quotients of  affine Nash groups by their Nash subgroups are naturally affine Nash manifolds, and when the Nash subgroups are normal, the quotients are naturally affine Nash groups. The following theorem is an analogue of Chevalley's theorem for affine Nash groups.
\begin{thmd}\label{main}
Let $G$ be a connected affine Nash group. Then there exists a unique connected normal almost linear Nash subgroup $H$ of $G$ such that $G/H$ is an abelian Nash manifold.
\end{thmd}

The semi-direct product of two affine Nash groups is clearly an affine Nash group. We will also prove the following theorem, which asserts that  the semi-direct product of two almost linear Nash groups is again an almost linear Nash group.

\begin{thmd}\label{semip}
Let $G$ and $H$ be almost linear Nash groups, with a Nash action $G\times H\rightarrow H$ of $G$ as Nash automorphisms of $H$. Then the semi-direct product $G\ltimes H$ is an almost linear Nash group.
  \end{thmd}

Theorem \ref{semip} is  basic to the structure theory of almost linear Nash groups, as developed in \cite{Sun}. The authors' original motivation of this note is just to provide a proof of Theorem \ref{semip}.  As an analogue of Theorem \ref{semip} for linear algebraic groups, we know that the semi-direct product of two linear algebraic groups is still a linear algebraic group. This is proved by the reason that an algebraic group is linear if and only if it is affine as an algebraic variety.  But affine Nash groups are not necessarily almost linear. Therefore the aforementioned simple reason for linear algebraic groups does not prove Theorem \ref{semip}. It seems to the authors that Theorem \ref{semip} is not easy to prove within the framework of \cite{Sun} for almost linear Nash groups, and it is necessary to go to the broader setting of affine Nash groups.

  \section{Preliminaries on algebraic groups}
 For later use, we recall some well-known facts concerning algebraic groups in this section. The reader is referred to \cite{Mil2} for more details. Let $\rk$ be a field.
  As usual, an algebraic variety over $\rk$ is defined to be a separated, geometrically reduced scheme over $\rk$ of finite type. Let $\mathsf G$ be an algebraic group over $\rk$,
  namely a group object in the category of algebraic varieties over $\rk$. Then $\mathsf G$ is automatically smooth over $\rk$. Recall that when $\rk$ has characteristic zero,
  every group scheme over $\rk$ is geometrically reduced (\cf \cite[Section V.3, Corollary 3.9]{Per}).
   Let $\mathsf H$ be an algebraic subgroup of $\mathsf G$, namely,  a geometrically reduced closed subgroup
scheme of $\mathsf G$.

\begin{prp}\label{quoa}
   The fppf quotient  $\mathsf G/ \mathsf H$ is represented by a smooth algebraic variety over $\rk$.
  \end{prp}
  \begin{proof}
  This is proved in \cite[Exp VIA, Thm 3.2]{Gr}
  \end{proof}

 Recall that a fppf quotient is a universal geometric quotient, and a universal geometric quotient is a universal categorical quotient (\cf \cite[Chapter IV]{GM}).
 Specifically, the quotient $\mathsf G/ \mathsf H$ of Proposition \ref{quoa} is a universal categorical quotient as well as a universal geometric quotient. When $\mathsf H$ is normal in $\mathsf G$, the quotient $\mathsf G/ \mathsf H$ is naturally an algebraic group over $\rk$.

  Recall that $\mathsf G$ is said to be linear if it is isomorphic to a closed algebraic subgroup of the general linear group $\GL(n)_{/\rk}$, for some $n\geq 0$.
  It is well know that $\mathsf G$ is linear if and only if it is affine (\cf \cite[Theorem 3.4]{Wat}).  We say that $\mathsf G$ is an abelian variety if it is complete and connected.
  When this is the case, $\mathsf G$ is commutative, and is projective as a variety. Note that every connected algebraic group over $\rk$ is geometrically connected (\cf Lemma 32.5.14 of the stacks project).

  The following fundamental result is due to Chevalley and is known as Chevalley's theorem.

  \begin{thmp}\label{mainch} (\cf \cite[Theorem 1.1]{Con})
Assume that $\rk$ is a perfect field, and $\mathsf G$ is connected. Then there exists a unique connected, normal, linear,  algebraic subgroup $\mathsf L$ of $\mathsf G$ such that the quotient $\mathsf G/ \mathsf L$ is an abelian variety.
\end{thmp}

 Chevalley's theorem has the following consequence.

\begin{corp}\label{maincr}
The smooth algebraic variety $\mathsf G/ \mathsf H$ is quasi-projective.
\end{corp}

 \begin{proof}
 It is proved in \cite[Corollary 1.2]{Con}  that $\mathsf G$ is quasi-projective. The same proof (with a slight modification)  shows that $\mathsf G/ \mathsf H$ is also quasi-projective.
  \end{proof}

Recall the following  well-known result.
  \begin{lemp}\label{disj}
  Let $\mathsf G_1$ be a connected linear algebraic group, and let $\mathsf G_2$ be an abelian variety, both defined over $\rk$.
  Then there is no non-trivial algebraic homomorphism from $\mathsf G_1$ to $\mathsf G_2$, and no non-trivial algebraic homomorphism from $\mathsf G_2$ to $\mathsf G_1$.
  \end{lemp}
\begin{proof}
It is proved in \cite[Lemma 2.3]{Con} that there is no non-trivial algebraic homomorphism from $\mathsf G_1$ to $\mathsf G_2$. Since every regular function on a geometrically connected complete algebraic variety is constant, there is no non-trivial algebraic homomorphism from $\mathsf G_2$ to $\mathsf G_1$.
\end{proof}

  \begin{prp}\label{exactalg}
  Let
\[
  1\rightarrow \mathsf G'\stackrel{i}{\rightarrow} \mathsf G\stackrel{\pi}{\rightarrow} \mathsf G''\rightarrow 1
\]
 be a sequence of algebraic homomorphisms of algebraic groups over
$\rk$. Assume that it is exact, namely, $\pi$ is faithfully flat,
and $i$ induces an isomorphism from $\mathsf G'$ onto the
scheme-theoretic kernel of $\pi$.  Then the followings hold true.

(a) The algebraic group $\mathsf G$  is linear if and only if both $\mathsf G'$ and $\mathsf G''$ are so.

(b) The algebraic group $\mathsf G$  is complete if and only if both $\mathsf G'$ and $\mathsf G''$ are so.

\end{prp}

\begin{proof}

This is also well known. We sketch a proof for the lack of reference. It is obvious that if $\mathsf G$ is linear, then $\mathsf G'$ is linear, and if $\mathsf G$ is complete, then $\mathsf G'$ is complete.

Note that the diagram
\[
   \begin{CD}
          \mathsf G'\times_{\Spec \rk} \mathsf G @>(x,y)\mapsto xy >> \mathsf G\\
            @V (x,y)\mapsto y VV           @VV \pi V\\
            \mathsf G @>\pi>> \mathsf G''
  \end{CD}
\]
is Cartesian. Since being affine (or proper) is local on the base for the fppf topology, we know that if $G'$ is affine or complete, then $\pi$ is affine or proper, respectively. Therefore, if both $\mathsf G'$ and $\mathsf G''$ are affine (or complete), then so is $\mathsf G$.

If $\mathsf G$ is complete, then $\mathsf G''$ is also complete since $\pi$ is surjective (\cf \cite[Property 7.6]{Mil}).

Now assume that $\mathsf G$ is linear, and it remains to show that $\mathsf G''$ is also linear. Without loss of generality we assume that $\rk$ is algebraically closed and $\mathsf G$ is connected. Then $\mathsf G''$ is also connected. Let $\mathsf A$ be an abelian variety over $\rk$. Then $\pi$ induces an injective homomorphism
\[
  \Hom(\mathsf G'', A)\rightarrow \Hom(\mathsf G, A).
\]
By Lemma \ref{disj}, the group $\Hom(\mathsf G, A)$ is trivial. Therefore $\Hom(\mathsf G'', A)$ is also trivial. Since $\mathsf A$ is arbitrary, Chevalley's theorem implies that $\mathsf G''$ is linear.
\end{proof}

We remark that an algebraic homomorphism of algebraic groups is fully faithful if and only if it is surjective (\cf \cite[Fact 6.30]{Mil2}). Moreover, with the notation as in Proposition \ref{exactalg}, $\pi$ induces an isomorphism form $\mathsf G/\mathsf G'$ onto $\mathsf G''$ (\cf \cite[Theorem 6.27]{Mil2}).

We will also need the following elementary lemma.

\begin{lemp}\label{open}
Let $\mathsf G$ be an algebraic group over $\rk$. Let $\mathsf G_0$ be an open algebaic subgroup of $\mathsf G$. Then $G$ is linear (or complete) if and only if $\mathsf G_0$ is so.
\end{lemp}
\begin{proof}
The lemma is obvious when $\rk$ is algebraically closed. The general case is easily reduced to this case.
\end{proof}

\section{Algebraizations}

For every equi-dimensional smooth algebraic variety $\mathsf X$ over
$\R$, $\mathsf X(\R)$ and $\mathsf X(\C)$ are naturally Nash
manifolds. Likewise, for every algebraic group $\mathsf G$ over
$\R$, $\mathsf G(\R)$ and  $\mathsf G(\C)$ are naturally a Nash
groups. We introduce  the following definition.

\begin{dfnl}\label{defalgb}
Let $G$ be a Nash group. An algebraization of $G$ is an algebraic group $\mathsf G$  over $\R$, together with a Nash homomorphism $G\rightarrow \mathsf G(\R)$ which has a finite kernel and whose image is Zariski dense in $\mathsf G$.
\end{dfnl}
Here $\mathsf G(\R)$ is naturally identified with a subset of the
underlying topological space of the scheme $\mathsf G$. Similar
identifications will be used later on without further explanation.
We call the homomorphism $G\rightarrow \mathsf G(\R)$ of Definition
\ref{defalgb} the algebraization homomorphism of the algebraization.

\begin{lem}\label{nz0}
Let $\mathsf G$ be an algebraic group over $\R$. Let $G$ be a Nash
subgroup of $\mathsf G(\R)$ which is Zariski dense in $\mathsf G$.
Then $G$ is  open in the Nash group $\mathsf G(\R)$.
\end{lem}
\begin{proof}
Note that $G$ is Zariski dense in $\mathsf G(\R)$. Therefore by
\cite[Theorem 3.20]{Cos}, we know that the Nash groups $G$ and
$\mathsf G(\R)$ has the same dimension. Hence $G$ is open in
$\mathsf G(\R)$.
\end{proof}

 Lemma \ref{nz0} implies that the Lie algebra of a Nash group is identical to the Lie algebra of every algebraization of it.

Recall that  every real projective space is affine as a Nash
manifold (cf. \cite[Theorem 3.4.4]{BCR}). Therefore, by Proposition
\ref{quoa}, for every algebraic group  $\mathsf G$  over $\R$,
$\mathsf G(\R)$ is  an affine Nash group. The following result is
crucial to this note.

\begin{thml}\label{main30}
A Nash group is affine if and only if it has an algebraization.
\end{thml}
\begin{proof}
The ``if" part follows from the fact that a finite fold cover of an
affine Nash manifold is an affine Nash manifold (\cf
\cite[Proposition III.1.7]{Sh} and \cite[Proposition 2.4]{Sun}). The
``only if" part is proved by E. Hrushovski and A. Pillay (see
\cite[Theorem B]{HP1} and \cite[Proposition 3.1]{HP2}).
\end{proof}

The following lemma asserts that  algebraizations of an affine Nash group is unique up to coverings.

\begin{lem}\label{algec}
Let $G$ be an affine Nash group. Let $\mathsf G_1, \mathsf G_2$ be
two algebraizations of $G$. Then there exist an algebarization
$\mathsf G_3$ of $G$, and two surjective algebraic homomorphisms
$\mathsf G_3 \rightarrow \mathsf G_1$ and $\mathsf G_3\rightarrow
\mathsf G_2$ with finite kernels such that the the diagram
\[
  \xymatrix{
   G \ar[r] \ar[rd] \ar[d] &\mathsf G_1(\R) \\
  \mathsf G_2(\R)  &\mathsf G_3(\R) \ar[l] \ar[u] .
   }
\]
commutes. Here the three arrows starting from $G$ are the
algebraization homomorphisms.
\end{lem}
\begin{proof}
Take $\mathsf G_3$ to be the Zariski closure of $\{(\phi_1(x),
\phi_2(x))\mid x\in G\}$ in $\mathsf G_1\times_{\Spec \R} \mathsf
G_2$, where $\phi_1: G\rightarrow \mathsf G_1(\R)$ and $\phi_2:
G\rightarrow \mathsf G_2(\R)$ denotes the algebraization
homomorphisms. Then the lemma easily follows.
\end{proof}

\begin{lem}\label{algehom}
Let $\mathsf G$ be an algebraization of an affine Nash group $G$. Let $G'$ be an affine Nash group with a Nash homomorphism $\varphi: G'\rightarrow G$. Then there exists an algebraization $\mathsf G'$ of $G'$, together with an algebraic homomorphism $\underline{\varphi}: \mathsf G'\rightarrow \mathsf G$ such that the diagram
\[
   \begin{CD}
           G' @>\varphi>> G\\
            @V\phi' VV           @VV \phi V\\
            \mathsf G'(\R) @>\underline{\varphi} >> \mathsf G(\R)
  \end{CD}
\]
commutes. Here $\phi'$ and $\phi$ denote the algebraization
homomorphisms.
\end{lem}
\begin{proof}
Let $\mathsf G_0'$ be an algebraization of $G'$. Take $\mathsf G'$ to be the Zariski closure in $\mathsf G_0'\times_{\Spec \R} \mathsf G$ of the image of the homomorphism
\[
   G'\rightarrow \mathsf G_0'(\R)\times \mathsf G(\R), \quad x\mapsto (\phi_0'(x), \phi(\varphi(x))),
\]
where $\phi_0': G'\rightarrow \mathsf G_0'(\R)$ denotes the
algebraization homomorphisms. Let $\underline{\varphi}$ be the
restriction to $\mathsf G'$ of the projection homomorphism $\mathsf
G_0'\times_{\Spec \R} \mathsf G\rightarrow \mathsf G$. Then the
lemma follows.
\end{proof}

\begin{lem}\label{alge}
Let $\mathsf G$ be an algebraization of an affine Nash group $G$. Then $G$ is almost linear if and only if $\mathsf G$ is linear as an algebraic group.
\end{lem}
\begin{proof}
The ``if" part is trivial. To prove the ``only if" part of the lemma, we assume that $G$ is almost linear. Take a Nash homomorphism $\varphi: G\rightarrow \GL_n(\R)$ with finite kernel. Denote by $\mathsf G_1$ the Zariski closure of $\varphi(G)$ in the algebraic group $\GL(n)_{/\R}$. Put $\mathsf G_2:=\mathsf G$, and let $\mathsf G_3$ be as in Lemma \ref{algec}. Since $\mathsf G_1$ is linear, part (a) of Proposition \ref{exactalg} implies that $\mathsf G_3$ is linear, which further implies that $\mathsf G_2$ is linear.
\end{proof}

\begin{lem}\label{alge2}
Let $\mathsf G$ be an algebraization of a connected affine Nash group $G$. Then $G$ is an abelian Nash manifold if and only if $\mathsf G$ is an abelian variety.
\end{lem}
\begin{proof}
Let $\phi: G\rightarrow \mathsf G(\R)$ denote the algebraization
homomorphism. First note that $\mathsf G$ is connected since it is
the closure of a connected subset.

Let $\mathsf L$ be the connected linear algebraic subgroup of
$\mathsf G$ as in Theorem \ref{mainch}.  Then $\phi^{-1}(\mathsf
L(\R))$ is an almost linear Nash group in $G$. If $G$ is an abelian
Nash manifold, then $\phi^{-1}(\mathsf L(\R))$ is finite. Hence
$\mathsf L$ is trivial and $\mathsf G=\mathsf G/\mathsf L$ is an
abelian variety.

Let $H$ be a connected almost linear Nash subgroup of $G$. Denote by $\mathsf H$ the Zariski closure of $\phi(H)$ in $\mathsf G$, which is a connected algebraic subgroup of $\mathsf G$.
By Lemma \ref{alge}, $\mathsf H$ is linear.   If $\mathsf G$ is an abelian variety, then Lemma \ref{disj} implies that $\mathsf H$ is trivial.
Hence $H$ is trivial. This proves that $G$ is an abelian Nash manifold.

\end{proof}

\begin{lem}\label{alge3}
Let $\mathsf G$ be an algebraization of an affine Nash group $G$.
Then $G$ is
complete  if and only if $\mathsf G$ is complete as an
algebraic variety.
\end{lem}
\begin{proof}
Note that $G$ is complete if and only if its identity connected
component $G^\circ$ is an abelian Nash manifold. Similarly, by Lemma
\ref{open}, $\mathsf G$ is complete if and only if its identity
connected component $\mathsf G^\circ$ is an abelian variety.  Since
$\mathsf G^\circ$ is an algebraization of $G^\circ$, the lemma is a
direct consequence of Lemma \ref{alge2}.
\end{proof}

\begin{corl}\label{abncc}
Every abelian Nash manifold is commutative and compact as a Lie group.
\end{corl}
\begin{proof}
Let $G$ be an abelian Nash manifold. Let $\mathsf G$ be an
algebraization of $G$, which is an abelian variety by Lemma
\ref{alge2}. Then $\mathsf G(\R)$ is a commutative compact Lie
group. Since $G$ is connected and there is a Lie group homomorphism
with finite kernel from $G$ onto the identity connected component of
$\mathsf G(\R)$, we know that $G$ is commutative and compact.

\end{proof}

\begin{corl}\label{disjn}
  Let $G_1$ be a connected almost linear Nash group, and let $G_2$ be an abelian Nash manifold.
  Then there is no non-trivial Nash homomorphism from $G_1$ to $G_2$, and no non-trivial algebraic homomorphism from $G_2$ to $G_1$.
  \end{corl}
\begin{proof}
In view of Lemmas \ref{alge} and \ref{alge2}, the corollary easily
follows from Lemma \ref{algehom} and Lemma \ref{disj}.
\end{proof}

\section{Chevalley's theorem for affine Nash groups}

We begin with the following proposition, which defines the quotients of affine Nash groups.
\begin{prp}\label{quotient0}
Let $G$ be an affine Nash group, and let $H$ be a Nash
subgroup of it. Then there exists a unique Nash structure on the
quotient topological space $G/H$ which makes the quotient map
$G\rightarrow G/H$ a submersive Nash map. With this Nash structure,
$G/H$ becomes an affine Nash manifold, and the left translation map
$G\times G/H\rightarrow G/H$ is a Nash map. Furthermore, if $H$ is a
normal Nash subgroup of $G$, then the topological group $G/H$
becomes an affine Nash group under this Nash structure.
\end{prp}

\begin{proof}
Let $\mathsf G$ be an algebraization of $G$, with the algebraization
homomorphism $\phi: G\rightarrow \mathsf G(\R)$. Let $\mathsf H$
denote the Zariski closure of $\phi(H)$ in $\mathsf G$. Then by
Corollary \ref{maincr}, $\mathsf G(\C)/\mathsf H(\C)=(\mathsf
G/\mathsf H)(\C)$ is the set of $\C$-points of a smooth
quasi-projective algebraic variety over $\R$. Therefore $\mathsf
G(\C)/\mathsf H(\C)$ is naturally an affine Nash manifold. Note that
all arrows in
\[
   G\stackrel{\phi}{\rightarrow} \mathsf G(\R)\rightarrow \mathsf G(\C)\rightarrow \mathsf G(\C)/\mathsf H(\C)
\]
are Nash maps. Therefore, the composition, which is denoted by $\phi'$, is also a Nash map.  By \cite[Proposition 5.53]{BOR}, the homogeneity (under the action of $G$) of $\phi'(G)$ implies that $\phi'(G)$  is a Nash submanifold of $\mathsf G(\C)/\mathsf H(\C)$. Note that $\phi'$ induces a finite-fold covering map
\be\label{covergh}
   G/H\rightarrow \phi'(G).
\ee By \cite[Proposition 2.4]{Sun}, there is a unique Nash structure
on the topological space $G/H$ which makes the map \eqref{covergh} a
submersive Nash map. By  \cite[Lemma 2.7]{Sun}, we know that the
quotient map $G\rightarrow G/H$ is a Nash map, with the
aforementioned Nash structure on $G/H$. It is a consequence of
\cite[Theorem 3.62]{Wa} that the  quotient map $G\rightarrow G/H$ is
submersive, and it is implied by
\cite[Proposition III.1.7]{Sh} that the Nash manifold $G/H$ is affine.  This proves the existence of the desired  Nash structure
on $G/H$. All other assertions of the proposition easily follows
from \cite[Lemma 2.3]{Sun}.

\end{proof}

When $G$ is almost linear, Proposition \ref{quotient0} is formulated and proved in \cite[Proposition 1.2]{Sun}. By Proposition \ref{quotient0}, the quotients of an affine Nash group by a Nash subgroup of it is naturally an affine Nash manifold, and is naturally an affine Nash group if the Nash subgroup is normal.

Now we are prepared to formulate and prove Chevalley's theorem for affine Nash groups.

\begin{thmp}
Let $G$ be a connected affine Nash group. Then there exists a unique connected normal almost linear Nash subgroup $L$ of $G$ such that $G/L$ is an abelian Nash manifold.
\end{thmp}

\begin{proof}
Let $\mathsf G$ be an algebraization of $G$, with the algebraization homomorphism $\phi: G\rightarrow \mathsf G(\R)$. Using Chevalley's theorem, let $\mathsf L$ be the unique connected normal linear algebraic subgroup of $\mathsf G$ such that $\mathsf G/\mathsf L$ is an abelian variety. Let $L$ be the identity connected component of $\phi^{-1}(\mathsf L(\R))$, which is a connected normal Nash subgroup of $G$. Note that $\mathsf L$ is an algebraization of $L$ and $\mathsf G/\mathsf L$ is an algebraization of $G/L$. Therefore Lemma   \ref{alge} and Lemma \ref{alge2} imply that $L$ is an almost linear Nash group, and $G/L$ is an abelian Nash manifold. This proves the ``existence" part of the theorem.

To prove the uniqueness, let $L'$ be a  connected normal almost linear Nash subgroup of $G$ such that $G/L'$ is an abelian Nash manifold. Let $\mathsf L'$ denote the Zariski closure of $\phi(L')$ in $\mathsf G$, which is a connected normal algebraic subgroup of $\mathsf G$. Note that $\mathsf L'$ is an algebraization of $L'$ and $\mathsf G/\mathsf L'$ is an algebraization of $G/L'$.  Lemma   \ref{alge} and Lemma \ref{alge2} again imply that $\mathsf L'$ is a linear algebraic group, and $\mathsf G/\mathsf L'$ is an abelian variety. Then Chevalley's theorem implies that $\mathsf L'=\mathsf L$. Therefore $L'=L$ since they have the same Lie algebra.

\end{proof}

The following result is an analogue of Proposition \ref{exactalg} in the setting of affine Nash groups.

  \begin{thmp}\label{exactnash}
  Let
\be\label{exact}
  1\rightarrow  G'\stackrel{i}{\rightarrow}  G\stackrel{\pi}{\rightarrow}  G''\rightarrow 1
\ee
be a sequence of Nash homomorphisms of affine Nash groups. Assume that it is exact as a sequence of  abstract groups. Then the followings hold true.

(a)  The  affine Nash group $G$  is almost linear if and only if both $G'$ and $G''$ are so.

(b) The affine Nash group $G$  is complete  if and only if both $G'$ and $G''$ are so.
\end{thmp}

\begin{proof}
Let $\mathsf G''$ be an algebraization of $G''$. By Lemma \ref{algehom}, we have an algebraization $\mathsf G$ of $G$, together with an algebraic homomorphism $\underline{\pi}: \mathsf G\rightarrow \mathsf G''$ such that the diagram
\[
   \begin{CD}
           G @>\pi>> G''\\
            @VVV           @VVV\\
            \mathsf G(\R) @>\underline{\pi} >> \mathsf G''(\R)
  \end{CD}
\]
commutes, where the vertical arrows are the algebraization
homomorphisms. Denote by $\mathsf G'$ the kernel of
$\underline{\pi}$, which is an algebraic group over $\R$ and fits to
an exact sequence \be\label{exact}
  1\rightarrow \mathsf G'\stackrel{i}{\rightarrow} \mathsf G\stackrel{\pi}{\rightarrow} \mathsf G''\rightarrow 1
\ee
of algebraic groups. Note that an open algebraic subgroup of $\mathsf G'$ is an algebraization of $G'$. In view of Lemma \ref{open}, Lemmas \ref{alge} and \ref{alge2} imply that $G'$ is almost linear or complete if and only if $\mathsf G'$ is respectively linear or complete. Lemmas \ref{alge} and \ref{alge2} also imply that $G''$ (or $G$) is almost linear or complete if and only if $\mathsf G''$ (or $\mathsf G$, respectively) is respectively linear or complete. Therefore, the proposition is a consequence of Proposition \ref{exactalg}.

\end{proof}

Theorem \ref{semip} is clearly a special case of part (a) of Theorem
\ref{exactnash}.

\end{document}